\documentclass{amsproc}

\usepackage{}

\newtheorem{theorem}{Theorem}[section]
\newtheorem{lemma}[theorem]{Lemma}

\theoremstyle{definition}
\newtheorem{definition}[theorem]{Definition}

\theoremstyle{remark}
\newtheorem{remark}[theorem]{Remark}

\numberwithin{equation}{section}

\newtheorem{proposition}[theorem]{Proposition}
\newtheorem{corollary}[theorem]{Corollary}

\def\a{\alpha}

\def\e{\epsilon}

\def\R{{\Bbb R}}  
\def\N{{\Bbb N}}  
\def\Z{{\Bbb Z}}  
\def\Q{{\Bbb Q}}  

\def\H{{\Bbb H}}

\let\cal=\mathcal

 \def \e {{\varepsilon}}

 \def \d {{\delta}}
 \def \a {{\alpha}}

 \def \ba {\begin{array}}
 \def \ea {\end{array}}

 \def \cC {{\cal C}}

 \def \cF {{\cal F}}
 \def \cG {{\cal G}}
 \def \cH {{\cal H}}

 \def \cN {{\cal N}}
 \def \cP {{\cal P}}

 \newcommand{\be}{\begin{equation}}
 \newcommand{\ee}{\end{equation}}

\newcommand{\bea}{\begin{eqnarray}}
 \newcommand{\eea}{\end{eqnarray}}

\begin{document}

 \title[Aperiodic complexities, dimensions and Diophantine approximation]
{ The aperiodic complexities and connections to dimensions and Diophantine approximation }


\author{Viktor Schroeder}
\address{Mathematisch-naturwissenschaftliche Fakult\"at, Universit\"at Z\"urich, Winterthurerstrasse 190, 8057 Z\"urich, Switzerland}
\curraddr{}
\email{viktor.schroeder@math.uzh.ch}
\thanks{}

\author{Steffen Weil}
\address{School of Mathematical Sciences, Tel Aviv University, Tel Aviv 69978, Israel}
\curraddr{}
\email{steffen.weil@math.uzh.ch}
\thanks{}



\date{}

\begin{abstract}
In their earlier work (Ergodic Th. Dynam. Sys., $34$: $1699 -1723$, $10$ $2014$), the authors introduced the so called $F$-aperiodic orbits of a dynamical system on a compact metric space $X$, which satisfy a quantitative condition measuring its recurrence and aperiodicity.
Using this condition we introduce two new quantities $\cF$, $\cG$, called the \emph{aperiodic complexities}, 
of the system and establish relations between $\cF$, $\cG$ with the topology and geometry of $X$.
We compare them to well-know complexities such as the box-dimension and the topological entropy.
Moreover, we connect our condition to the distribution of periodic orbits and  
we can classify an $F$-aperiodic orbit of a point $x$ in $X$ in terms of the collection of the introduced approximation constants of $x$.
Finally, we discuss our results for several examples, in particular for the geodesic flow on hyperbolic manifolds.
For each of our examples there is a suitable model of Diophantine approximation and we classify $F$-aperiodic orbits in terms of Diophantine properties 
of the point $x$. 
As a byproduct, we prove a `metric version' of the closing lemma in the context of CAT(-1) spaces.
\end{abstract}

\maketitle

\section{Introduction and main results}
Given a compact metric space $X=(X,d)$ and a continuous map $T: X \to X$,
there are various connections between the dynamics of the discrete dynamical system $(X,T)$ and the topology and (global) geometry of $X$. We refer to \cite{PesinClimenhaga} for a good reference on this topic.
We introduce two new quantities $\cF$, $\cG \geq 0$ which we call the \emph{aperiodic complexities} and which are defined as follows.
Denote by $\N_0 = \{0\} \cup \N$ and by $\N^* \equiv \N \cup \{\infty\}$.
Recall that the \emph{Bowen metric} $d_l$ of \emph{length} $l \in \N_0$ on $X$ is given by
\be
\label{BowenMetric}
	d_l(x, y) \equiv \max_{0\leq i \leq l}  d(T^ix, T^i y).
\ee
Clearly we have $d=d_0$ and $d_l$ induces the same topology as $d$.
Given a point $x \in X$ and $\e>0$, define its  \emph{return time} or \emph{shift time}
\be
\nonumber
	s_l(x, \e)  \equiv \inf \{ s \in \N : d_l(T^sx,x) < \e \} \in \N^*
\ee
with respect to the metric $d_l$.
To the  orbit $\cal{T}(x) \equiv \{ T^n(x) : n\geq 0\}$ of $x$,
assign the  \emph{shift function} $F_x^l: (0, \infty) \to \N$ 
defined by 
\be
\label{DefFx}
	F_x^l(\e) \equiv \min\{ s_l(T^nx, \e) : n\in \N_0 \} \in \N,
\ee
which exists by compactness of $X$. %
The orbit $\cal{T}(x)$ is called  \emph{$F_x^l$-aperiodic}.
Note that $F_x^l(\e) = 1$ for $\e> \text{diam}(X)$.

\begin{remark}
For a non-increasing function $F:(0, \infty) \to [1, \infty)$ the authors introduced the concept of  \emph{$F$-aperiodic} orbits $\cal{T}(x)$ in their earlier work \cite{SchroederWeil} (see also Definition \ref{DefAperiodic}).
Note that if $\cal{T}(x)$ is $F$-aperiodic, then it is also $F'$-aperiodic for $F'\leq F$, where here and in the following $F'\leq F$ if $F'(\e)\leq F(\e)$ for all $\e>0$. 
Therefore the function $F_x^l$ from above can be characterized as
\be
\nonumber
	F_x^l(\e) = \sup\{ F(\e) : F \in \text{AP}(x) \},
\ee
where $AP(x)$ is the set of functions $F$, such that the orbit $\cal{T}(x)$ is $F$-aperiodic for the dynamical system $(X, d_l, T)$. 
In this sense, $F_x^l$ is the optimal function measuring the aperiodicity of the orbit $\cal{T}(x)$.
\end{remark}

\noindent The function $F_x^l$ quantitatively measures the complexity of the \emph{whole} orbit in terms of the recurrence times (with respect to the metric $d_l$).
Clearly, if $x$ is a periodic point, then  $F_x^l$ is bounded by its period.
Moreover, it follows from Section \ref{Prelims} that, in a suitable setting,
we have for a generic point $x$ that $F_x^l$ is bounded (which yields $\cF_x=0$ below).

In the following we are interested in the special points with unbounded shift functions and positive exponential growth rates.
In this regard, note that each $F_x^l$ is non-increasing, whereas $F_x^m \leq F_x^l$ for $m \leq l$.
We then define the aperiodic complexities $\cF_x$ and $\cG_x$ of the orbit $\cal{T}(x)$  by
\be
\label{ExpRate}
	\cF_x \equiv \limsup_{\e \to 0} \frac{\log(F_x^0(\e))}{- \log(\e)}, \ \ \ \ \cG_x \equiv \lim_{\e \to 0} \limsup_{l \to \infty} \frac{\log(F_x^l(\e))}{l},%
\ee
and finally define the \emph{aperiodic complexities} of the system $(X,T)$ as
\be
\label{AperiodicComplexity}
	\cF = \cF(X,T) \equiv \sup_{x \in X} \cF_x, \ \ \ \ \cG=\cG(X,T) \equiv \sup_{x \in X} \cG_x.
\ee
Note that $\cF$ and $\cG$ satisfy certain properties which typically hold for dimensions, see Section \ref{SectionDimensions}. 
Moreover, the authors established positive lower bounds for $\cF$ and $\cG$ in \cite{SchroederWeil} for several examples; we recall the results of this work in Section \ref{Apps}.

However, it turns out that orbits $\cal{T}(x)$ with a positive exponential growth rate $\cF_x $, respectively $\cG_x$, 
(or even unbounded functions $F_x^l$) are extremely rare and proving their existence is delicate in most cases. 
The purpose of this paper is to show in addition that such an orbit turns out to be `special' in the following sense:

On the one hand, the existence of a complicated orbit requires `space'.
We make this intuition precise by showing that the  box-dimension of $X$ and respectively the topological entropy can be estimated from below by the aperiodic complexities.

On the other hand a very aperiodic orbit  turns out to be \emph{bounded} with respect to \emph{every} periodic point in a `uniform' sense:
 the orbit avoids a \emph{critical neighborhood} (defined via $F_x^l$) for every periodic point.
We will discuss the existence of aperiodic orbits for several examples in Section \ref{Apps}. 
In particular, we consider the geodesic flow on compact hyperbolic manifolds as a central example.
In these examples the critical neighborhoods and the distribution of periodic points can be related to a suitable setting of Diophantine approximation.
Moreover, we can classify the dynamics of the orbits with respect to the functions $F_x^l$  in terms of Diophantine properties of the point $x$.
In order to do so we need to establish suitable versions of the closing lemma for the respective dynamical systems.
This in particular leads to a `metric version' of the closing lemma in the context of CAT(-1) spaces (see Proposition \ref{ClosingLemma}).

More precisely, the first aim of this paper is to establish upper bounds for $\cF$ and $\cG$ 
in terms of well-known complexities of the space $X$ and the dynamical system $(X,T)$ in the spirit of the following results.

\begin{theorem} 
\label{ThmOne}
Let $\dim_B(X)$ and $h(X, T)$ denote the upper box dimension of $X$ and the topological entropy of $(X, T)$, respectively.   
Then
\be
\label{Thm1Bounds}
	\cF\leq \text{dim}_B(X), \ \ \ \ \cG \leq h(X, T).
\ee	
\end{theorem}

\noindent  
We refer to Section \ref{SectionAperiodic} for definitions and proofs. Note that for the examples in Section \ref{Apps} we actually have equality in \eqref{Thm1Bounds} due to the results of \cite{SchroederWeil}, showing that these estimates are optimal.

In the special case that $M$
denotes a closed Riemannian manifold, consider its geodesic flow $\phi^t$ (or rather its time-one map) on the unit tangent bundle $SM$.
Then Manning \cite{Manning}
connects the \emph{volume entropy $\lambda$} of $M$ with the topological entropy $h(\phi^t)$ of the geodesic flow $\phi^t$ on $SM$ by
\be
\nonumber
	\lambda \leq h(\phi^t),
\ee
with equality in the case that $M$ has non-positive sectional curvature. 
In Sections \ref{SectionVolumeEntropy} and \ref{GeodResults} we provide further details and show the following, which is also a consequence of Theorem \ref{ThmOne} and Manning's result.

\begin{theorem}
\label{ThmIntro}
When $M$ is of non-positive curvature, then
\be
\nonumber
	\cG \leq \lambda
\ee
Equality holds in the case that $M$ is of constant non-positive curvature with injectivity radius $i_M> \log(2)$%
\footnote{ We however believe that the condition $i_M> \log(2)$ is not necessary. }
.
\end{theorem}

\begin{remark}
There are many more relations between the dynamics of the system $(X,T)$ and complexities of the space $X$ in the literature, such as the Hausdorff-dimension.
For instance, while our condition is based on the \emph{whole} orbit and in most cases on `non-typical' orbits,
 Boshernitzan already  established a relation between the Hausdorff-dimension to the quantitative rate of recurrence for almost all points in $X$ (for a suitable measure and setting); see \cite{Bosh}  for details.
We refer to Section \ref{SectionHausdorff} for further discussion and results.
\end{remark}

The second aim of the paper is to provide a characterization of aperiodic orbits in terms of how they avoid certain neighborhoods of periodic points.
Therefore we establish a connection between aperiodicity and the distribution of periodic orbits.
Denote by $\cal{P}_T$ the set of $T$-periodic points in $X$ and assume there exists a periodic point $x_p \in \cal{P}_T$ of 
period $p\in \N$ (the index  of $x_p$ will stand for its period $p$).
When the shift function $F_x^0$ of the orbit $\cal{T}(x)$ is unbounded, then  $\cal{T}(x)$ does not intersect
the open neighborhood
\be
\nonumber
	\cal{N}_{x_p}(\e) \equiv  B_{d}(x_p,  \e) \cap T^{-p}\big(B_{d}(x_p,\e)\big)
\ee
 of $x_p$ for every  small enough $\e>0$, see Propositon \ref{PropositionCriticalNbhd}. 
We may also say that the orbit  $\cal{T}(x)$ is \emph{bounded} with respect to the obstacle $x_p$.
In this case we have
\be
\label{ApproxConst}
	c_{x_p}(x) \equiv \inf\{ \e>0 : \cal{T}(x) \cap \cal{N}_{x_p}(\e) \neq \emptyset \} >0,
\ee
which we call the \emph{approximation constant} of $x$ with respect to the periodic point $x_p$.

\begin{remark}
Let $\mu$ be an ergodic Borel probability measure. If $x_p \in $ supp$(\mu)$ for some $x_p \in \cal{P}_T$ then  $c_{x_p}(x)=0$ for $\mu$-almost every $x$. With respect to the periodic point $x_p$, we may call a point $x$ \emph{well approximable} when $c_{x_p}(x)=0$, and otherwise \emph{badly approximable}.
Indeed, the \emph{shrinking target property}, due to Hill and Velani \cite{HillVelani}, considers
more generally a  sequence of nested measurable sets $A_n \subset X$ (in our case $\cal{N}_{x_p}(\e_n)$ with $\e_n\to 0$)
and is interested in the properties of the points in $ X$ whose orbit hits $A_n$ for infinitely many times $n$.
Such points are called \emph{well approximable} (with respect to $\{A_n\}$) by analogy with Diophantine approximation.
Conversely, a point for which the orbit avoids some set $A_n$ for some  $n=n(x)$ may be called \emph{badly approximable}.
\end{remark}

Let us remark that  a suitable version of the closing lemma will hold for all the examples considered in Section \ref{Apps}.
This in turn will imply a quantitative property of the system $(X, T)$ stating that recurrence is `caused by periodic orbits', 
which we call \emph{$\d$-closing property} with respect to a non-decreasing function $\d: (0, \infty) \to (0, \infty) $ (see Definitions \ref{ClosingProperty}).
Given a non-increasing unbounded function $F : (0, \infty) \to \N$, we can define versions of its inverse function given by the \emph{quantile functions} $F^{\leftarrow}$, $F^{\rightarrow} : \N \to (0, \infty)$,
\be
\label{QuantileFct}
	F^{\leftarrow}(s) \equiv \sup\{ \e>0 : F(\e) > s\}, \ \ \  F^{\rightarrow}(s) \equiv \inf\{ \e>0 : F(\e) \leq s\}.
\ee
Clearly, if $F: (0, \infty)\to (0, \infty)$ is continuous and bijective, then $F^{\leftarrow}(s) = F^{\rightarrow}(s)  = F^{-1}(s)$.
We classify $F_x^0$-aperiodic orbits $\cal{T}(x)$ in terms of the collection of the approximation constants $\{c_{x_p}(x)\}$;
see Theorem \ref{ThmBounded}  for further details and proofs.

\begin{theorem}
Let $F: (0,  \infty) \to \N$ be a non-increasing and unbounded function.
If $\cal{T}(x)$ is $F_x^0$-aperiodic with  $F_x^0 \geq F$,
then for every periodic point $x_p$ we have
\be
	\nonumber
	c_{x_p}(x) \geq F^{\leftarrow}(p) /2.
\ee
Conversely, if the system $(X,T)$ satisfies the $\d$-closing property  and we have  
\be
\nonumber
	c_{x_p}(x)>  F^{\rightarrow}(p)
\ee
for every periodic point $x_p$, 
then $\cal{T}(x)$ if $\tilde F_x^0$-aperiodic with $\tilde F_x^0 \geq F\circ \d$.
\end{theorem}

Note that a similar result holds for the classification of $F_x^l$-aperiodic orbits, see Theorem \ref{BoundedLength} in Section \ref{DA}.
\\

\noindent \emph{Acknowledgements.} 
S.W. expresses his sincere gratitude to Jean-Claude Picaud. Several results were motivated or improved by numerous discussions with him.
He was partially supported by the ERC starter grant DLGAPS $279893$.
Both authors acknowledge the support by the Swiss National Science Foundation (grant no. $135091$) and the referee for the detailed comments and
corrections.


\section{Quantitatively Aperiodic and Recurrent Orbits}
\label{SectionAperiodic}

\subsection{Preliminaries}
\label{Prelims}
Recall that $(X,d)$ is a compact metric space and $T : X \to X$ is a continuous transformation. 
Moreover let $\mu$ be a  Borel probability measure on $X$ for which $T$ is measure-preserving; see \cite{Walters}.
A point $x\in X$ is called \emph{periodic} (with respect to $T$) if 
there exists an integer $p\in \N$, called a \emph{period} of $x$, 
such that $T^px=x$. We write $x_p$ to indicate that $x_p$ is periodic and has \emph{primitive} period $p$.
Denote by $\cal{P}_T$ the $T$-invariant set of $T$-periodic points of $X$.
A point is called \emph{aperiodic}, if it is not periodic.
A point $x \in X$ is \emph{recurrent} with respect to $T$, if for any 
$\e>0$ its return time $s(x, \e) \in \N$ is finite.
Periodic points are obviously recurrent with $s(x, \e)$ bounded by the period.
We recall that, by the Poincar\'e-recurrence theorem, $\mu$-almost every point is recurrent,
and the set $\cal{R}_T$ of recurrent points is $T$-invariant.
However, the return time $s(T^n x,\varepsilon)$ at time $n$ can differ from $s(x,\varepsilon)$ in general.

We are interested in a quantitative condition on recurrence and aperiodicity of whole orbits.
Given $\e>0$, we ask independently from the time for a
lower bound on the \emph{shift} $s$ such that $T^{n+s} x$ is allowed to be $\varepsilon$-close to $T^nx$:

\begin{definition} 
\label{DefAperiodic}
For a non-increasing function $F : (0, \infty) \to [1, \infty)$ a point $x \in X$ is 
called \emph{$F$-aperiodic}  if for every  $\varepsilon>0$ and every shift $s\in \N$, we have
\be 
\label{DefFAperiodic}
	d(x,T^s x)<\e 	\Longrightarrow s \geq F(\e).
\ee 
The orbit $\cal{T}(x)$ is \emph{$F$-aperiodic}, if it is $F$-aperiodic at every time $n \in \N_0$, that is, if $T^nx$ is $F$-aperiodic.%
\footnote{ Note that we changed the terminology from \cite{SchroederWeil}.  }
\end{definition}

\begin{remark}
Note that \eqref{DefFAperiodic} reads that the return time $s_0(x, \e) \geq F(\e)$,
or, that $d(x, T^sx) \geq \e$ whenever $s < F(\e)$. 
Moreover, if $F$ is continuous then $s \geq F(d(x, T^sx))$.
If in addition $F$ is invertible, we have $d(x, T^sx) \geq  F^{-1}(s)  \equiv \tilde F(s)$.
These conditions could serve as an alternative definition of \eqref{DefFAperiodic}.
\end{remark}

\noindent 
Recall that  every orbit $\cal{T}(x)$ is $F_x^0$-aperiodic with $F_x^0$ defined in \eqref{DefFx}.
We emphasize that although we called the condition `$F$-aperiodic', 
a periodic orbit  is 
$F$-aperiodic for a suitable bounded function $F$. 
We therefore view the growth rate of $F$ also as a measure for its aperiodicity.

In terms of the topological entropy of $(X,T)$, the notion of $F$-aperiodic orbits will turn out to be unsuitable and we need to adapt the definition to the specific setting.
Recall the Bowen metric $d_l$ of length $l\in \N_0$, defined in \eqref{BowenMetric}.
Let $G : \N_0 \times (0, \infty) \to [1, \infty)$ be a two-parameter function, where for
\be	
\nonumber
	G(l, \e) \equiv G_l (\e) \equiv G_{\e}(l)
\ee
we assume that the restricted functions $G_l : (0, \infty) \to [1, \infty)$ are non-decreasing
and $G_{\e} : \N_0 \to [1, \infty)$ are non-increasing, for every $l \in \N_0$ and $\e>0$.

\begin{definition} 
\label{Def2}
Given the length $l\in \N_0$, a point $x \in X$ (respectively an orbit $\cal{T}(x)$) is called \emph{$G_l$-aperiodic} 
if it is $G_l$-aperiodic in the metric space $(X, d_l)$.
Finally, $\cal{T}(x)$ is \emph{$G$-aperiodic}, if it is $G_l$-aperiodic for every length $l\in \N_0$.
\end{definition}

\noindent Recall that every orbit $\cal{T}(x)$ is $G$-aperiodic for the function $G(l, \e) \equiv F_x^l(\e)$ with $F_x^l$ defined in \eqref{DefFx}.
Moreover note that $\cal{T}(x)$ is $G$-aperiodic
if for each $n \in \N_0$ the following is satisfied for the point $y=T^nx$: 
for every length $l\in \N_0$, for every  $\e>0$ and every shift $s\in \N$, we have
\be \nonumber
	d_l(y,T^{s} y) = \max_{0\leq i \leq l}  d(T^iy, T^{i+s} y) <\varepsilon \ \  	\Longrightarrow \ \ s \geq G_l(\e)= G(l,\varepsilon) .
\ee

Assume in the following that $T$ admits periodic points and let $F:(0, \infty) \to \N$ be non-decreasing and unbounded and recall the definition of $F^{\leftarrow}$ in \eqref{QuantileFct}.
Then the question of existence of $F$-aperiodic orbits is related to the distribution of periodic orbits.
Indeed,  an $F$-aperiodic orbit avoids periodic orbits in the following quantitative sense.
Let $x_p \in \cal{P}_T$ and define the nonempty open set
\be
\label{CriticalNbhd}
	\cal{N}_{x_p}(F) \equiv \cal{N}_{x_p}(F^{\leftarrow}(p)/2) =  B(x_p,   F^{\leftarrow}(p) /2) \cap T^{-p}\big(B(x_p, F^{\leftarrow}(p) /2)\big),
\ee
called the \emph{critical neighborhood} of $x_p$ with respect to the function $F$.

\begin{proposition}
\label{PropositionCriticalNbhd}
The set $\cal{N}_{x_p}(F)$ cannot contain any $F$-aperiodic point; in particular, an $F$-aperiodic orbit $\cal{T}(x)$ avoids $\cal{N}_{x_p}(F)$ and cannot be dense.

Moreover, if $\mu$ is a $T$-ergodic Borel probability measure and the support of $\mu$ contains a periodic point $x_p \in \cal{P}_T$,
then the set $\{ x \in X : \cal{T}(x) $ is $F$-aperiodic$ \}$  is a $\mu$-null set. 
\end{proposition}

Similar results hold for $G$-aperiodic (respectively $G_l$-aperiodic) orbits by defining a suitable critical neighborhood as in \eqref{PenetrationCalF} below.
Thus, in the case when $T$ is ergodic with respect to $\mu$,  while a `typical' orbit is dense, 
this shows that an $F$-aperiodic orbit is non-typical when $F$ is unbounded.

\begin{proof}
For every point $y\in \cal{N}(x_s, F) = B(x_s, F^{\leftarrow}(s)/2) \cap T^{-p} (B(x_s, F^{\leftarrow}(s)/2))$  
we have by the triangle inequality that $d(x,T^{s}y) <  F^{\leftarrow}(s)  \equiv  \e$.
This gives a return time $s_0=s(y,  \e) \leq s$.
However, for any $F$-aperiodic point we would have
\be
\nonumber
	s_0 \geq F(\e) = F( F^{\leftarrow}(s)) > s,
\ee 
showing that no point in $\cal{N}_{x_p}( F)$ can be $F$-aperiodic. 
Thus we see that an $F$-aperiodic orbit $\cal{T}(x)$ must avoid $\cal{N}_{x_p}(  F)$.

For the second part, note that by definition the set $S \equiv \{ x \in X : \cal{T}(x) $ is $F$-aperiodic$ \}$  is $T$-invariant, hence $\mu(S) \in\{0,1\}$ by ergodicity.
If $x_p \in  $ supp$(\mu)$, then $\mu(\cal{N}_{x_p}(F))>0$ and the proposition follows since $S$ is disjoint to $\cal{N}_{x_p}(F)$.
\end{proof}


\subsection{Bounded orbits and classification of aperiodic orbits}
\label{DA}
In this section, we relate the question of the existence of $F$-aperiodic (respectively $G$-aperiodic) orbits again to the  distribution of periodic points for the general setup.
In particular, this motivates  a classification of $F$-aperiodic orbits in terms of a collection of quantities which measure how `bounded' the orbits are with respect to every periodic orbit.

More precisely, assume in the following that $T$ admits periodic points and that $F:(0, \infty) \to \N$ is non-increasing and unbounded.
Given a periodic point $x_p \in \cal{P}_T$, 
recall the definitions of the critical neighborhood $\cN(x_p, F)$ in \eqref{CriticalNbhd} and the approximation constant $c_{x_p}(x)$ in \eqref{ApproxConst}.  
Then define the set of \emph{bounded points} with respect to the periodic point $x_p $ by
\bea
	\nonumber
	\textbf{Bounded}_{x_p} &\equiv& \{x \in X:  \exists \e>0 \text{ such that } \cal{T}(x) \cap \cal{N}_{x_p}( \e) = \emptyset \} 
	\\ \nonumber
	&=& \{x \in X : c_{x_p}(x)>0\}.
\eea
Moreover, we define the \emph{$F$-bounded points} with respect to $x_p$ by  
\bea
	\label{PenetrationF}
	\textbf{Bounded}_{x_p}( F) &\equiv& \{x \in X: \cal{T}(x) \cap \cal{N}_{x_p}( F)  = \emptyset \} 
	\\ \nonumber
	&\subset& \{x \in X: c_{x_p}(x) \geq F^{\leftarrow}(p)/2\} \subset \textbf{Bounded}_{x_p}.
\eea

\begin{remark}
\label{Hurwitz}
Note that the `Hurwitz-constant' with respect to $x_p$, given by
\be
\nonumber
	\cH_{x_p} \equiv \sup_{x \in X} c_{x_p}(x) \leq \text{diam}(X),
\ee
is finite.
The existence of $F$-aperiodic orbits gives lower bounds for the collection $\{\cH_{x_p} : x_p \in \cP_T\}$, and conversely, this collection determines pointwise upper bounds for  functions  $F$ such that $F$-aperiodic orbits can exist;
indeed, we must have for such $F$ that
\be
\nonumber
	F(2\cH_{x_p}) \leq p.
\ee
Compare this with Example \ref{Torus} below, that is to the classical theory of Diophantine approximation.
\end{remark}

We now consider the condition that recurrence is caused by periodic orbits in the following sense:

\begin{definition}
\label{ClosingProperty}
We say that the system $(X, d, T)$ satisfies the \emph{$\d$-closing property} 
with respect to the non-decreasing function $\delta : (0, \infty) \to (0, \infty)$,
if, whenever $d(x, T^sx) < \e$ for some $\e>0$ with $\d(\e) \leq  \text{diam}(X)$, some point $x \in X$ and $s\in \N$, 
 there exists a periodic point $x_s \in \cal{P}_T$ of period $s$ such that
\be
	\nonumber
  	x \in \cal{N}(x_s, \d(\e));	\ \ \text{ equivalently, }	\ \ d( x, x_s) < \d(\e) \ \ \ \text{and} \ \ \ d(T^sx, x_s) < \d(\e).
\ee
\end{definition}

\begin{remark}
Note that if $\d(\e)>\text{diam}(X)$ then the above condition is empty and the closing property holds automatically. 
Hence, also in view of Theorem \ref{ThmBounded} below, we are interested in functions $\d$ which are as small as possible.
\end{remark}

Assuming the $\d$-closing property, we can classify $F$-aperiodic orbits 
as follows.

\begin{theorem}
\label{ThmBounded}
Let $\cal{T}(x)$ be $F$-aperiodic. Then, necessarily
\be
\label{Bad}
	x \in \bigcap_{x_p \in \cal{P}_T}  \textbf{Bounded}_{x_p}( F).
\ee
Conversely, suppose that $(X, d ,T)$ satisfies the $\d$-closing property  and for $x \in X$ we have 
\be
\nonumber
	c_{x_p}(x) > F^{\rightarrow}(p)
\ee
for all $x_p \in \cal{P}_T$.
Then $\cal{T}(x)$ is $\tilde F$-aperiodic for the function $\tilde F(\e)= F(\d(\e))$.
\end{theorem}

\begin{proof}
The first part follows immediately from Proposition \ref{PropositionCriticalNbhd}.
For the second part, assume that $d(T^i x, T^{i+s}x) < \e $.
By the $\d$-closing property,  if $\d = \d(\e) \leq $ diam$(X)$, there exists a periodic point $x_s$ of period $s$ such that we have $T^ix \in \cal{N}(x_s, \d)$.
But since $c_{x_p}(x) > F^{\rightarrow}(s)$, hence $T^ix \not \in \cal{N}(x_s, F^{\rightarrow}(s)) $,
we must have $\d> F^{\rightarrow}(s)$.
This shows 
\be
\nonumber
	F(\d(\e)) \leq F(F^{\rightarrow}(s) ) \leq s.
\ee
If $\d(\e)>  \text{diam}(X)$, then $F(\d(\e))\leq 1 \leq s$ anyway,
finishing the proof.
\end{proof}

\begin{remark}
We remark that the closing property will be satisfied in all  examples considered in Section \ref{Apps}.
Moreover, in the theory of `bounded orbits' one often considers the concept of Schmidt's game (see \cite{Schmidt} and references thereof)
in order to capture properties of a set \textbf{Bounded} of bounded points.
In our examples in Section \ref{Apps}, every set $\textbf{Bounded}_{x_p}$, $x_p \in \cal{P}_T$, will be a Schmidt winning set,
and as such, this implies  the intersection $\cap_{x_p}\textbf{Bounded}_{x_p}$ to be dense and nonempty.
However, \eqref{Bad} is a stronger condition since, due to the restriction of $F$, we do not consider the full set $\textbf{Bounded}_{x_p}$.
\end{remark}


Concerning $G$-aperiodic orbits,
assume in the following that for every $\e>0$ we have a non-decreasing and unbounded
function $ G_{\e} = G(\e, \cdot) : \N_0 \to \N$.
For the function $G_{\e}$ define the following quantile functions $G_{\e}^{\leftarrow}$, $G_{\e}^{\rightarrow} : \N \to \N_0$ by
\be
\nonumber
	G_{\e}^{\leftarrow}(s) \equiv \min\{ l \in \N_0 : G_\e(l)\geq  s\}, \ \ \ G_{\e}^{\rightarrow}(s) \equiv \max\{ l \in \N_0 : G_\e(l)\leq  s\}.
\ee

We first define the penetration length in neighborhoods of periodic orbits.
With respect a given point $x_0 \in X$ and $\e>0$,
the \emph{penetration length} of a point $y \in X$ is given by $\frak{p}_{x_0, \e}(y) \equiv 0$ if $y \not \in B(x_0, \e)$ 
and otherwise by
\bea
\label{DefPenetrationLength}
	\frak{p}_{x_0, \e}(y) &\equiv& \sup\{l \in \N_0 : y \in B_{d_l}( x_0, \e) \} +1
	\\ \nonumber
	&=&  \sup\{l \in \N_0 : d(T^i y, T^i x_0) < \e \text{ for } 0\leq i \leq l \} +1 \in \N^*.
\eea

Analogously to \eqref{PenetrationF},  given $\e>0$,
we define the \emph{$G$-bounded points} with respect to a given periodic point $x_p \in \cal{P}_T$ 
by  
\bea
\label{PenetrationCalF}
	\textbf{Bounded}_{x_p}(G_{ \e}) &\equiv& \{x \in X: \cal{T}(x) \cap B_{d_{p + G_{2\e}^{\leftarrow}(p)} }(x_p, \e)= \emptyset \}
		\\ \nonumber
			&=&   \{x \in X:  \frak{p}_{x_p, \e}(T^nx) \leq p + G_{2\e}^{\leftarrow}(p) +1, \forall n\geq 0\}.
\eea
For the new situation, we need to adjust the version of the closing property.

\begin{definition}
\label{StrongClosingProperty}
We say that the system $(X, d, T)$ satisfies the \emph{strong $\d$-closing property} 
with respect to $\e>0$ and a non-decreasing function $\d=\delta_{\e} : \N_0 \to \R^+$,
if, whenever $d_l(x, T^sx) < \e$ for some point $x \in X$, shift $s\in \N$ and length $l\in \N_0$,
 there exists a periodic point $x_s \in \cal{P}_T$ of period $s$  such that
\be
\nonumber
  	 \frak{p}_{x_s, \e}(x) \geq s+ \delta_\e(l) +1; \ \  \text{ equivalently } \ \ x \in B_{d_{s+\delta_\e(l)}}(x_s, \e). 
\ee
\end{definition}

Under the assumption of the strong $\d$-closing property, 
we give the following interpretation. 

\begin{theorem}
\label{BoundedLength}
Let $\cal{T}(x)$ be a $G$-aperiodic orbit.
Given $\e>0$, for every periodic point $x_s$ of period $s\in \N$,
we have that,
the penetration lengths of $T^n(x)$ are  bounded by 
\be
\nonumber
	\frak{p}_{x_s, \e}(T^nx)  \leq s + G^{\leftarrow}_{2\e}(s) +1,
\ee
for all $n\in \N_0$.
In particular, for every $\e>0$, necessarily
\be
\label{BadCalF}
	x \in \bigcap_{x_s \in \cal{P}_T} \textbf{Bounded}_{x_s}( G_{\e}).
\ee 
Conversely, suppose that $(X, T, d)$ satisfies the strong $\d$-closing property for the function $\d=\delta_{\e}$ 
 and that there is a point $x$ satisfying
\be
\nonumber
	\frak{p}_{x_s, \e}(T^nx)  \leq s + G^{\rightarrow}_{\e}(s) +1,
\ee
for every $\e>0$. 
Then $\cal{T}(x)$ is $G'$-aperiodic for the function $G'(\e, l) =G(\e, \d_{\e}(l))$.
\end{theorem}

\begin{proof}
Let $x_s \in \cal{P}_T$ and assume that $d_l(T^nx, x_s) < \e$ for the time $n\in \N_0$ and and for a  length $l\geq s+1$. 
Hence,
\bea
	\nonumber
	d_{l-s}(T^nx, T^{n+s}x) &\leq& d_{l-s}(T^nx,  x_s) + d_{l-s}( x_s, T^{n+s}x ) 
	\\ \nonumber
	&=&  d_{l-s}(T^nx,  x_s) + d_{l-s}(T^{s} x_s, T^{n+s}x ) < 2 \e
\eea
Thus, since $\cal{T}(x)$ is $G$-aperiodic, we have $s \geq G(l-s, 2 \e) =  G_{2\e}(l-s)$.
This shows $l \leq s+ G^{\leftarrow}_{2\e}(s)$ and hence $\frak{p}_{x_s, \e}(T^nx)  \leq s + G^{\leftarrow}_{2\e}(s) +1$.

Conversely, assume that $d_l(T^nx, T^{n+s}x) < \e$ for some time $n\in \N_0$, shift $ s\in \N$ and length $l\in \N_0$.
By the strong $\d$-closing property, there exists a periodic point $x_s$ of period $s$ such that 
\be
\nonumber
	\frak{p}_{x_s, \e}(T^nx) \geq s+\d_{\e}(l) +1. 
\ee
Hence, using  assumption \eqref{BadCalF}, we have
\be
\nonumber
 	s+\d_{\e}(l) +1\leq s + G^{\rightarrow}_{2\e}(s) + 1.
\ee
This shows $s \geq G(\e, \d_{\e}(l))$, finishing the proof.
\end{proof}


\subsection{Dimensions and the aperiodic complexity}
\label{SectionDimensions}

Recall that every orbit $\cal{T}(x)$ is $F$-aperiodic for the function $F=F_x^0$, respectively $G$-aperiodic for the function $G(l, \e)=F_x^l(\e)$.
Recall also the definitions of the aperiodic complexities $\cF$ and $\cG$ in \eqref{AperiodicComplexity} via the exponential growth rates $\cF_x$ and $\cG_x$, respectively, given in \eqref{ExpRate}.

Given a nonempty subset $Y \subset X$, we equip $Y$ with the induced metric.
The following properties of the aperiodic complexity $\cF$ are easily verified.
\begin{itemize}
\item[1.] \emph{Monotonicity.} If $U \subset V \subset X$ are $T$-invariant, then $\cF(U,T\lvert_{U}) \leq \cF(V,T\lvert_{V})$.
\item[2.] \emph{Countable Stability.}	If $Y_n\subset X$, $n\in \N$, are $T$-invariant, then 
	\be
		\nonumber
		\cF(\cup_n Y_n, T\lvert_{\cup Y_n}) = \sup\{\cF(Y_n, T\lvert_{Y_n}) : n\in \N\}.
	\ee
\item[3.] \emph{Bilipschitz Invariance.} If $f: X \to f(X)=Y$ is a bi-Lipschitz map,
			then 
			\be
				\nonumber
				\cF(X, T) = \cF(Y, f \circ T\circ f^{-1}).
			\ee
\item[4.] \emph{Product.} $\max\{ \cF(X_1 , T_1),  \cF(X_2 , T_2)\}\leq  \cF(X_1 \times X_2, T_1 \times T_2) $, with the product metric on $X_1 \times X_2$. 
\end{itemize}
The same properties hold for the complexity $\cG$.

We now show that  $F$-aperiodic orbits affect the geometry of $X$ in the large,
meaning that the existence of an $F$-aperiodic, respectively a $G$-aperiodic orbit for suitable functions $F$ and $G$, 
requires that certain complexities of $X$, respectively of $(X, T)$, must be positive. 
For a good reference concerning the complexities of a space and of dynamical systems, we refer to \cite{PesinClimenhaga}.
\\

\subsubsection{Box dimension}

For $\e>0$ let $N(X,\e)$ denote the number of a maximal $\e$-separated set in $X$. 
Then the \emph{upper box dimension} (\cite{Falconer}) is given by
\be \nonumber
	\dim_B(X) = \limsup_{\e\to 0} \frac{\log(N(X, \e))}{-\log(\e)}.
\ee

\begin{proposition} 
[\cite{SchroederWeil}, Lemma 2.2]
\label{ThmBoxDim} 
For every $x \in X$ we have
\be
\nonumber
	 \cF_x \leq \dim_B(X).
\ee 
In particular, $\cF \leq \dim_B(X)$.
\end{proposition}

\begin{proof} 
Let $\varepsilon>0$ and $F=F_x^0$.
We claim that the set
\be
\label{Separated}
	\{T^n x : n=0, \dots, \lfloor F(\e) \rfloor \} \subset \cal{T}(x)
\ee
gives an $\e$-separated set; hence $N(X, \e)\geq F(\e)$.
In fact, for  every $0\leq s_1 < s_2\leq F(\varepsilon)$
we have $d(T^{s_1}x, T^{s_2}x) \geq \e$ since $s_2-s_1 < F(\varepsilon)$.
Taking the exponential growth rates finishes the proof.
\end{proof}


\subsubsection{Remarks on the Hausdorff-dimension}
\label{SectionHausdorff}
Let $\mu$ be a finite Borel measure on $X$ satisfying 
$\mu(B(x,r))  \leq a \cdot r^{\delta}$ for all sufficiently small $0< r \leq r_0$ and constants $a, \delta>0$.
It is well known that the exponent $\delta$ is a lower bound for the Hausdorff-dimension dim$_H(X)$ of $X$ (see \cite{Falconer}, Proposition 4.9).
Under the requirement that $\mu$ is $\delta$-Ahlfors regular (or Ahlfors-David regular), that is  $a r^{\delta} \leq \mu(B(x,r))\leq b r^{\delta}$ for all $x\in $ supp$(X)$ and $0<r<r_0$,
we next show that
\be
\nonumber
	\cF \leq \delta \leq \text{dim}_H(X).
\ee

\begin{proposition}
\label{ThmHausdorffDim}
Let $\mu$ be a  finite Borel-measure which is $\delta$-Ahlfors regular with supp$(\mu)=X$.
Then, for every $x \in X$, we have 
\be
\nonumber
	\cF_x \leq \delta.
\ee
In particular, $\cF \leq \delta$.
\end{proposition}

\begin{proof}
We know from \eqref{Separated} that the points $\{T^nx\}$ are $\e$-separated  for $n\leq F(\varepsilon)$, where $F=F_x^0$.
Thus, for $\e \leq r_0$, we have
\bea
\nonumber
	\mu(X) &\geq& \sum_{n=0}^{\lfloor F(\e) \rfloor} \mu(B(T^nx, \e/2) 
	\\ \nonumber
	&\geq&   \lfloor F(\e) \rfloor \cdot \min\{ \mu(B(T^nx, \e/2) : 0\leq n \leq F(\e)\}	
	\geq a  \lfloor F(\e) \rfloor (\e/2)^{\delta}.
\eea
Applying the logarithm, dividing by $-\log(\e)$ and taking the exponential growth rates finishes the proof.
\end{proof}

\begin{remark}
Based on the work of Boshernitzan \cite{Bosh},  Barreira, Saussol \cite{BarreiraSaussol} related the lower recurrence rate of a point,
defined by 
\be
\nonumber
	\underline{R}(x) \equiv \liminf_{\e \to 0}  \frac{ \log( s(x,\e))}{- \log(\e)},
\ee
to the Hausdorff-dimension of $X$.
In fact, they showed that if $T$ is a Borel measurable transformation on the separable metric space $X$, $\mu$ is a $T$-invariant Borel probability measure on $X$,
then $\underline{R(x)} \leq d_{\mu}(x)$ for $\mu$-almost every point $x \in X$. 
Here, $d_{\mu}$ denotes the lower-pointwise dimension of the measure $\mu$, where $d_{\mu}(x)=\delta$ for all $x$ if $\mu$ is $\delta$-Ahlfors regular.
This gives an upper bound on the recurrence for almost every point in $X$,
whereas, for an $F$-aperiodic point $x$, we (in general) have $\underline{R(x)} \geq \cF_x$.
\end{remark}


\subsubsection{Topological Entropy}

Let $T$ be a continuous map on a compact metrisable space $X$.
Let  $d$ be any metric on $X$ inducing the same topology on $X$.
For $\e>0$ and $l\in \N$, let $N_T(l,\e)$ be the number of a maximal $\e$-separated net with respect to the Bowen metric $d_l$.
Then the \emph{topological entropy $h(T)$} of $T$ is defined by
\be
\nonumber
	h(T) = \lim_{\e \to 0} \limsup_{l \to \infty} \frac{1}{l} \log(N_T(l, \e)).
\ee
The following Proposition, together with Proposition \ref{ThmBoxDim} , proves Theorem \ref{ThmOne}.

\begin{proposition}
\label{ThmTopEnt}
For every $x \in X$ we have
\be
\nonumber
	\cG_x \leq h(T). 
\ee
In particular, $\cG \leq h(T)$.
\end{proposition}

\begin{proof}
For $\varepsilon>0$ and $l \in \N$ let $G(l, \e)=F_x^l(\e)$.
We claim that
\be
\label{SeparatedCalF}
	\{T^nx : n=0, \dots,  \lfloor G(l, \e) \rfloor  \}\subset \cal{T}(x)
\ee
gives a $\e$-separated set in the metric $d_l$; hence $N_T(l, \e)\geq G(l, \e)$.
In fact, for every $0\leq s_1 < s_2\leq G(l,\varepsilon)$, 
we have $d_l(T^{s_1}x, T^{s_2}x) \geq \varepsilon$ since $s_2-s_1 <  G(l, \e)$.
Taking the exponential growth rates finishes the proof.
\end{proof}


\subsubsection{Volume Entropy}
\label{SectionVolumeEntropy}
Let $M=(M,d_M)$ be a compact Riemannian manifold, let $\pi : SM \to M$ the footpoint projection, where $SM$ denotes the unit tangent bundle.
Denote by $\phi: SM \times \R \to SM$ the geodesic flow on $SM$.
Given a fixed point $o \in \tilde M$, the universal cover of $M$,
let $V(o,r) = $ vol$(B(o,r))$ denote the volume of the metric ball $B(o,r)$.
Note that that the limit, called \emph{volume entropy},
\be
\label{DefVolEnt}
	\lambda \equiv \lim_{r \to \infty} \frac{1}{r}\log(V(o,r))
\ee
exists and is independent of $o$. Recall that from \cite{Manning} we have $\lambda \leq h(\phi)$,
with equality in the case that $M$ has non-positive sectional curvature.

Now let $i_M>0$ denote the injectivity radius of $M$.
For $l \geq 1$, we define a new metric on $SM$ (see Manning \cite{Manning}) by
\be
\label{NewBowenMetric}
	d_l(v, w) \equiv \max_{0 \leq t \leq l} d_M(\pi\circ \phi^tv, \pi\circ \phi^tw)
\ee
which can be seen as the Bowen metric of length $l$ with respect to $d_1$.
We need to adjust the definition of $G$-aperiodic orbits to continuous time flows and to our setting.

\begin{definition}
Fix $0<\e_0< i_M/2$, $l_0\geq0$ and let $\varphi: \R^+ \to \R^+$ be an increasing function.
Given a vector $v_0 \in SM$, the geodesic $\gamma(t) \equiv \pi \circ \phi^t(v_0)$ is called a \emph{$\varphi$-aperiodic ray} in $M$ (with respect to the parameters $\e_0$ and $l_0$) if the following condition is satisfied:
for all times $t \geq 0$, all shifts $s>\e_0$ and  lengths $l \geq l_0$, we have
\be
\label{WeakBarF}
	d_l( \phi^t v_0, \phi^{t+s}v_0) \leq \e_0  \Longrightarrow s \geq \varphi(l).
\ee
\end{definition}

Note that, up to considering the time one-map $\phi^1$, the geodesic $\gamma$ is $\varphi$-aperiodic if and only if the orbit  $\phi^{t}(v_0)$ is $G$-aperiodic (for the metric $d_l$ above):

\begin{lemma}
\label{Umrechnung}
A ray $\gamma_v$ is $\varphi$-aperiodic if and only if the orbit $\phi^t(v)$ is $G$-aperiodic;
simply set $G( l, \e)= \varphi(l)$ for $\e\leq \e_0$ in both cases.
\end{lemma}

\noindent The proof of the Lemma is immediate and we are  interested in the following Theorem.

\begin{theorem}
\label{ThmVolEnt}
Let $M$ be of nonpositive curvature.
Given a vector $v_0$ satisfying \eqref{WeakBarF} for the function $\varphi$,
we have
\be
\nonumber
	 \limsup_{l \to \infty} \frac{1}{l} \log(\varphi(l)) \leq \lambda.
\ee
\end{theorem} 

We refer to \cite{Eberlein} for further background and details of the following.
In this setting, by nonpositive curvature, we have that the universal cover $\tilde M$ of $M$  is diffeomorphic to $\R^n$
and the (free) fundamental group $\Gamma \equiv \pi_1(M)$ of $M$ can be identified with a cocompact discrete subgroup of the isometry group of $\tilde M$
acting freely and properly discontinuously on $\tilde M$.
Note that the induced distance function $d=d_{\tilde M}$ on $\tilde M$ is convex.
The lifted flow, the footpoint projection and the lifted metrics $d_l$ on $S \tilde M$ are  denoted by the same symbols.

For later use we remark that $v_0 \in SM$ satisfies \eqref{WeakBarF} 
if and only if the following is true for any lift $\tilde v_0 \in S \tilde M$ of $v_0$:
Let $\gamma(t)= \pi \circ \phi^t(\tilde v_0)$  be the geodesic ray in $\tilde M$ determined by $\tilde v_0$.
Then for all times $t_0\in \R^+$, for all $ \psi \in \Gamma$, for all lengths $l \geq l_0$ and all shifts $s>\e_0$,
we have 
\be
\label{BeziehungLift}
	\max_{0\leq t \leq l}d(\gamma(t_0+s+t), \psi( \gamma(t_0 + t ))\leq \e_0 \Longrightarrow s\geq \varphi(l).
\ee
In fact, the left hand side of \eqref{BeziehungLift} reads $d_l(\phi^{t_0 + s}\tilde v_0, \phi^{t_0} \tilde v_0) < \e_0$
which is the case if and only if  $d_l(\phi^{t_0 + s} v_0, \phi^{t_0} v_0) \leq \e_0$  since $\e_0< i_M/2$.

\begin{proof}[Proof of Theorem \ref{ThmVolEnt} (which can be found in \cite{WeilDiss})]
Let $K$ be a compact fundamental domain of $\Gamma$ in $\tilde M$ of diameter $R$.
We may assume that the lift $\tilde v_0$ of $v_0$ is based at a point $p_0$ in $K$ and let $\gamma(t) \equiv \pi \circ \phi^t(\tilde v_0)$ be a geodesic ray in $\tilde M$ satisfying \eqref{BeziehungLift}.
Let $l$ be sufficiently large with respect to $R$, say $l \geq \max\{5R, l_0\}$.
Finally, consider the annulus $A= B_{l + 3R}(p_0) - B_{l-3R}(p_0)$ around $p_0$
and let $\cal{S}_K$ and $\cal{S}_A$  be maximal $\e_0/4$-separated sets of $K$ and $A$ respectively.
In particular, 
\be
\nonumber
	K \subset \bigcup_{x\in \cal{S}_K}B_{\e_0/2}(x), \ \ \ \ \ A \subset \bigcup_{x\in \cal{S}_A}B_{\e_0/2}(x).
\ee
Consider the vectors $\dot \gamma(t_i)= \phi^{t_i}(\tilde v_0)$ at the times $t_i \equiv 2Ri$, $i\in N_0$, along the geodesic $\gamma$.
For every $i\in \N_0$ we can find an isometry $\psi_i\in \Gamma$ such that 
the vector $v_i\equiv d\psi_i (\phi^{t_i}(\tilde v_0)) \in SK$
and since 
\be
\nonumber
	\lvert t_i - t_j \rvert \geq 2R = 2 \text{ diam}(K)
\ee
for $i \neq j$ we have that $\psi_i \neq \psi_j$.
Note that the endpoints of the geodesic of length $l$ determined by $v_i$ belong to $K$ and $A$ respectively, 
\be
\nonumber
	e^i_-\equiv \pi (v_i) = \psi_i(\gamma(t_i))  \in K, \ \ \  e^i_+ \equiv \pi(\phi^l(v_i))  =  \psi_i(\gamma(t_i + l)) \in A.
\ee 
In particular, there exists a pair $(x, y) \in \cal{S}_K \times \cal{S}_A$ such that $e^i_-\in B_{\e_0/2}(x)$ and $e^i_+\in B_{\e_0/2}(y)$.
We claim that, setting
\be
\nonumber
	N=\max\{i \in \N_0: 2Ri < \varphi(l)\},
\ee 
for every pair $(x,y) \in  \cal{S}_K \times \cal{S}_A$ there exists at most one pair $(e^i_-,e^i_+) \in B_{\e_0/2}(x) \times B_{\e_0/2}(y)$ when $i\leq N$.
Assuming the claim, this gives the estimate 
\be
\label{Bound1}
	\lvert \cal{S}_K \rvert \lvert \cal{S}_A \rvert = \lvert \cal{S}_K \times \cal{S}_A \rvert\geq N \geq (N+1)/2 \geq \varphi(l)/4R.
\ee

For the claim, assume that for $i< j\leq N$ we have $e^i_-, e^j_-\in B_{\e_0/2}(x)$ and $e^i_+, e^j_+\in B_{\e_0/2}(y)$.
Hence, $d(e^i_-,e^j_-)<\e_0$ and $d(e^i_+,e^j_+)<\e_0$.
Convexity of the distance function $d$ implies
\be
\nonumber
	 d_l(v_i, v_j) = \max_{0 \leq t \leq l} d(\pi(\phi^t(v_i)), \pi(\phi^t(v_j))) \leq \max\{d(e^i_-,e^j_-), d(e^i_+,e^j_+) \}\leq \e_0.
\ee
But setting $\psi \equiv \psi_j^{-1} \circ \psi_i$, where $\psi \neq id$ by the above, we get for the shift $s=t_j - t_i$,
\be
\nonumber
	\max_{0 \leq t \leq l} d(\gamma(t_i + s+  t), \psi( \gamma(t_i  +t) ) = \max_{0 \leq t \leq l} d( \psi_j( \gamma(t_j +t)), \psi_i(\gamma(t_i +  t))  =
	d_l(v_i, v_j) \leq \e_0.
\ee
Therefore, \eqref{BeziehungLift} %
implies
\be
\nonumber
	2RN\geq   2R(j-i) =  t_j - t_i \geq \varphi(l);
\ee
a contradiction to the definition of $N$, showing the claim.

On the other hand, let $a \equiv \min_{x\in K} $ vol$(B(x, \e_0/8) = \min_{x \in \tilde M}  $ vol$(B(x, \e_0/8)>0$.
Since $\cal{S}_A$ is $\e_0/4$-separated, for the $\e_0$-neighborhood $\cal{N}_{\e_0}(A)$ of $A$,
\be
\nonumber
	\text{vol}(\cal{N}_{\e_0}(A)) \geq \sum_{x \in \cal{S}_A}  \text{vol}(B(x, \e_0/8)) \geq \lvert \cal{S}_A \rvert a,
\ee
and similarly $\text{vol}(\cal{N}_{\e_0}(K))\geq \lvert \cal{S}_K \rvert a$.
Finally, using \eqref{Bound1}, this shows that
\bea
\label{IndependentEps}
	\text{vol}(B(p_0, l + 3R + \e_0)) &\geq& \text{vol}(\cal{N}_{\e_0}(A))) \geq \frac{a \varphi(l)}{4R  \lvert \cal{S}_K \rvert} 
	\\ \nonumber
	&\geq& \frac{a^2}{4R \text{ vol}(\cal{N}_{\e_0}(K))}  \varphi(l)  \equiv \bar c \cdot \varphi(l)
\eea
for a universal constant $\bar c = \bar c(M, \e_0) >0$.
Since $\lambda$ is independent of the point $p_0$,
taking the exponential growth rates as $l\to \infty$ finishes the proof.
\end{proof}

Formula \eqref{IndependentEps} also gives rise to several further corollaries.
In fact, it shows that every function $\varphi$, such that $\varphi$-aperiodic rays exist, is bounded by
\be
\nonumber
	\varphi(l) \leq \tfrac{1}{\bar c} \text{vol}(B(p_0, l + 3R + \e_0)).
\ee 

Morever, define the \emph{orbital counting function} $N_{\Gamma}(x)$  of $\Gamma$, $x \in \tilde M$ and $l \in \R^+$, by
\be
\nonumber
	N_{\Gamma}(x, l) \equiv \lvert \{ \psi \in \Gamma: d(x, \psi(x)) \leq l \} \rvert.
\ee
The quantity 
\be
\nonumber
	\d=\d_x \equiv \limsup_{l \to \infty} \frac{1}{l} \log(N_{\Gamma}(x, l))
\ee
is independent of the point $x$ and, in the case of constant negative curvature, in fact it equals the critical exponent
 $\d(\Gamma)$ of the Poincare series of $\Gamma$; see \cite{Nicholls} for further information.

\begin{corollary}
Assume there is a  $\varphi$-aperiodic geodesic ray.
Then
\be
\nonumber
	  \limsup_{l \to \infty} \frac{1}{l} \log(\varphi(l)) \leq \d(\Gamma).
\ee
\end{corollary}

\begin{proof}
By \eqref{IndependentEps}, it suffices to show that $\d(\Gamma) \geq \lambda$ (in fact, we even have $\lambda = \d(\Gamma)$).
Since $K$ is a fundamental domain, the ball $B(p_0, l)$ is covered by the sets $\psi(K)$ with $d(p_0, \psi(p_0)) \leq l$.
Moreover, $k \equiv \text{vol}(M)=\text{vol}(K)>0$, showing that
\be
\nonumber
			N_{\Gamma}(p_0, l) \geq \frac{\text{vol}(B(p_0, l)}{k}.
\ee
Taking the exponential growth rates finishes the proof.
\end{proof}


\section{Examples}
\label{Apps}

In this section, we consider  the geodesic flow on the torus, the Bernoulli shift and, as the central example, the geodesic flow on compact hyperbolic manifolds.
We discuss our results and set them in context to appropriate models of Diophantine approximation. 

\subsection{Geodesic flow and rotation on the torus}
\label{Torus}
Let $T^n=\R^n/\Z^n$ denote the flat torus where $d$ denotes the induced metric on $T^n\times \R^n$ of the product metric on $\R^n \times \R^n$.
For $x \in \R^n$, let $\bar x = x $ mod $\Z^n \in T^n$.
Consider the map on the tangent bundle $TT^n=T^n \times \R^n$ 
\be
	\phi : T^n \times \R^n \to T^n \times \R^n, \ \ \ (\bar x, \alpha) \mapsto ( \overline{x + \alpha}, 
	 \alpha),
\ee
which can be viewed as the time-one map of the geodesic flow on the flat torus.
Fixing the `direction' $\alpha$, this system actually corresponds to the rotation $R_{\alpha}$ on the torus $T^n$, however,
in order to discuss our conditions we chose to represent it in this form.

Since the topological entropy of the system equals zero,
there exists no $G$-aperiodic orbit for an exponentially increasing $G$ by Proposition \ref{ThmTopEnt}.
However, the situation is very different for $F$-aperiodic orbits and linked to classical Diophantine approximation.

Given a vector $\alpha \in \R^n$, let $\phi_{\alpha} : \R^n \to \R^n$, $\phi_{\alpha}(x) = x + \alpha$
be the lift of $R_{\alpha}$, that is to say of $\phi$ with direction restricted to $\alpha$.
Note that the distance of the  two vectors $(x,\alpha)$ and $(y, \alpha)$ equals the distance between their base points.
Then $\phi_{\alpha}$ acts on every orbit $\cal{T}(x)$ as an isometry; that is,
$\lVert \phi_{\alpha}^n(x) - \phi_{\alpha}^{n+s}(x) \rVert = \lVert x - \phi^{s}_{\alpha}(x) \rVert$.
Moreover an orbit $\phi^{\N_0}(x, \alpha)$ is $F$-aperiodic if and only if the orbit $\phi^{\N_0}(0, \alpha)$ is.
It thus suffices to look at the recurrence of the point $0 \in \R^n$.
In fact, for $\e>0$ sufficiently small and $s \in \N$, we have
\bea
\label{ClassicalDA}
	\bar d((\bar 0, \alpha), \phi^s(\bar 0, \alpha)) < \e \iff \exists p \in \Z^n : \e > \lVert p - \phi_{\alpha}^s(0) \rVert = \lVert  s \alpha - p  \rVert 
\eea
Recall that $\alpha$ is a badly approximable vector if there exists a constant $c=c(\alpha)>0$ with
\be
\nonumber
\label{ClassicalBad}
	\lVert  s \alpha - p  \rVert  \geq \frac{c}{s^{1/n}}
\ee
for all $s \in \N$, $p \in \Z^n$.
Thus, if $\a$ is badly approximable, we see that $s > c^n \e^{-n} \equiv F_{\alpha}(\e)$.
Hence $(\bar 0, \alpha)$ gives a $F_{\alpha}$-aperiodic orbit.

If conversely $(\bar 0, \alpha)$ is $F$-aperiodic for a function $F(\e) = c \e^{-n}$,
then \eqref{ClassicalDA} shows for every $s\in \N$, $p\in \Z^n$,
\be
\nonumber
	s \geq F( \lVert  s \alpha - p  \rVert ) = c \lVert  s \alpha - p  \rVert^{-n}
\ee
or in other words, $\lVert  s \alpha - p  \rVert \geq \frac{c^{1/n}}{s^{1/n}}$ and $\alpha$ is badly approximable with $c(\a) \geq c^{1/n}$.
This classifies $F_{\alpha}$-aperiodic orbits in terms of the approximation constant of a badly approximable  $\alpha \in \R^n$.

The exponential growth rate of $F_{\alpha}$ (defined above) equals $n$ which is the box-dimension of the tangent space $T_{\bar 0}T^n=\R^n$ at $\bar 0\in T^n$ 
which may be viewed as the space of directions.
This exponential growth rate is in fact the largest possible by Proposition \ref{ThmBoxDim},
which can also be seen by the following stronger result:
if $\cH$ denotes the Hurwitz-constant, 
\be
\nonumber
	\cH \equiv \sup\{ c(\a) : \a\in \R^n \text{ is badly approximable}\},
\ee
of the spectrum of badly approximable vectors
then, by the arguments above, no $F$-aperiodic orbit can exist for $F(\e)= c^n \e^{-n}$ with $c>\cH$.

Note also that periodic points $(\bar 0, \beta)$ of period $q$ correspond to rational vectors $p/q \in \Q^n$, $p \in \Z^n$, $q \in \N$,
and an $F$-aperiodic orbit must avoid each of them:
More precisely, restrict again to the space of directions $\{(\bar 0, \alpha) : \alpha \in \R^n\}$, identified with $\R^n$.
The critical neighborhood of a periodic point  $p/q \in \Q^n$ and for the function $F(\e) = c \e^{-n}$ with $F^{-1}(q) = (c/q)^{1/n}$ is readily determined  (similar to \eqref{ClassicalDA}) as 
\be
\nonumber
	 \cal{N}_{\frac{p}{q}}( F)=\{ \alpha \in \R^n : \lVert \alpha - p/q \rVert < \frac{c^{1/n}}{q^{1+1/n}} \}.
\ee
Moreover, we have the following \emph{Closing Lemma}, establishing the $\d$-closing property with $\delta(\e) = 2 \e$; see \eqref{ClosingProperty}.

\begin{lemma}
Assume that $d(\phi^s((\bar x, \alpha)), (\bar x, \alpha))< \e$ for a sufficiently small $\e>0$.
Then  there exists a vector $p \in \Z^n$ such that $\phi^s(\bar x, p/s) = (\bar x, p/s)$ with $d( (\bar x, \alpha), (\bar x, p/s) ) < \frac{1}{s} \e $
and  $d( \phi^s(\bar x, \alpha), (\bar x, p/s) ) < (1+\frac{1}{s}) \e $.
\end{lemma}

\begin{proof}
Let $x \in \R^n$ be a lift of $\bar x$.
The assumption reads that there exists a vector $p \in \Z^n$ with $\lVert x + s \alpha - (x + p) \rVert < \e$.
Clearly, $(\bar x, p/s)$ is periodic of period $s$.
Moreover, 
\be
\nonumber
	\lVert (x, p/s) - (x, \alpha) \rVert_{\R^n \times \R^n} = \lVert p/s - \alpha \rVert = 1/s  \lVert p - s\alpha \rVert  < \e/s
\ee
as well as 
\be
\nonumber
	\lVert (x+p, p/s) - (x + s\alpha, \alpha) \rVert_{\R^n \times \R^n} = \lVert (p - s \alpha, p/s - \alpha ) \rVert_{\R^n \times \R^n}< (1+\frac{1}{s}) \e,
\ee
finishing the proof.
\end{proof}


\subsection{Bernoulli shift}

For $n\geq 1$, let $\Sigma=\{1,\dots ,n\}^{\N}$ be the set of one-sided sequences in symbols from $\{1,\dots ,n\}$.
Let $T$ denote the shift and  $d$ be the metric on $\Sigma$ given by $d(w, w)\equiv0$ and 
\be
\nonumber
	d(w, \bar w)\equiv e^{-\min\{ i \geq 1: w(i) \neq \bar w(i) \}}  \ \ \ \text{ for } w \neq \bar w.
\ee
 
In our earlier work we showed the following existence theorem (stated for two sides sequences but also shown for one sided sequences).

\begin{theorem}[\cite{SchroederWeil}, Theorem 3.3]
\label{ExistenceWords}
Let $\varphi : \N \to \N$ be a non-increasing function such that
\be
\nonumber
	\limsup_{l \to \infty} \frac{1}{l} \log(\varphi(l)) < \log(n).
\ee
Then there exists a length $l_0 \in \N$ and a sequence $w \in \Sigma$ satisfying for every $l_0\leq l \in \N$ 
\be
\label{AperiodicWord}
	d(T^nw, T^{n+s}w) \leq e^{-(l+1)} \Longrightarrow s\geq \varphi(l).
\ee
\end{theorem}

\noindent
A sequence $w\in \Sigma$, satisfying \eqref{AperiodicWord} for a non-increasing function $\varphi : \N \to \N$ and $l_0 \in \N$, is called \emph{$\varphi$-aperiodic}.

We remark that due to the definition of $d$, we have that
\be
\nonumber
	d(w, w') \leq e^{-(l + k + 1)} \iff d_l(w, w') \leq e^{-(k+1)}.
\ee
Hence, the condition \eqref{AperiodicWord} can readily be translated into the following.

\begin{lemma}
Let $w$ be $\varphi$-aperiodic.
Then $w$ is $F$-aperiodic for the function $F(e^{-(l+1)}) = \varphi(l)$ and $G$-aperiodic for the function $G( e^{-(k+1)}, l)= \varphi(l+k)$.
\end{lemma}

Conversely, every $F$-aperiodic, respectively $G$-aperiodic sequence is $\varphi$-aperiodic for a suitable function $\varphi$.

Note that there is a natural measure $\mu$ on $\Sigma$ which is $\log(n)$-Ahlfors regular.
Using the above Lemma and Theorem \ref{ExistenceWords} (for the lower bounds) as well as Propositions \ref{ThmBoxDim}  and \ref{ThmTopEnt} (for the upper bounds), we obtain the following.

\begin{corollary}
$ \cF = $ dim$_H(\Sigma) = $ dim$_B(\Sigma) =\log(n) $ and $\cG = h(T) = \log(n) $.
\end{corollary}

To classify $\varphi$-aperiodic sequences in terms of periodic sequences, which lie densely in $\Sigma$, note that $(\Sigma, d, T)$ satisfies the strong $\d$-closing property, see \eqref{StrongClosingProperty}, for the function $\delta_1(l)=l$:

\begin{lemma}
Whenever $d(w, T^sw) \leq e^{-(l+1)}$, then there exists a periodic word $w_s$ of period $s$ such that $d(w, w_s) \leq e^{-(l+s+1)}$; hence $\frak{p}_{w_s, 1}(w) \geq s+l+1$.
\end{lemma}

\begin{proof}
In fact, let $w_s \in \Sigma$ be the periodic word of period $s$ such that $w_s(i)=w(i)$ for $i=1, \dots, s$.
Since $d(w, T^sw) \leq e^{-(l+1)}$, we also have $w_s(i)=w(i)$ for $i=s+1, \dots, s+l$.
The proof follows.
\end{proof}

Let $\varphi: (0, \infty)\to (0, \infty)$ be an increasing bijective function in the following.
We may reformulate the critical neighborhood of a periodic point given in \eqref{CriticalNbhd} as well as condition \eqref{Bad} 
and the strong $\d$-closing property to the setting of $\varphi$-aperiodic sequences.

\begin{proposition}
[\cite{SchroederWeil}, Proposition 3.4]
 \label{PeriodicDistance}
If $w\in \Sigma$ is $\varphi$-aperiodic (with say $l_0=0$), 
then for every periodic sequence $ w_s \in \Sigma$ of period $s$ and  for all times $n\in \N_0$ we have 
\be
\label{BernoulliCriticalNbhd}
	d(T^n w, w_s)\geq e^{- (s+\varphi^{-1}(s)  + 1)}.
\ee
Conversely, if $w$ satisfies \eqref{BernoulliCriticalNbhd}, then $w$ is $\varphi$-aperiodic.
\end{proposition}

\begin{proof} 
If $w$ is $\varphi$-aperiodic,  assume there exists $m\in \N$ such that $d(T^nw, w_s) =e^{-(l+1)}$ 
where we assume $l> s$ (otherwise the first statement follows).
Hence, 
\be
\nonumber
	w(n+1)\ldots w(n+l) = w_s(1)\ldots  w_s(l)
\ee
and, since $w_s$ is of period $s<l$,
we see that 
\be
\nonumber
	w(n+1)\ldots w(n+1+(l-s))=w(n+1+s)\ldots w(n +1 +l).
\ee
Thus, $d(T^nw, T^{n+s}w)\leq e^{-(l-s+1)}$ which implies $s\geq \varphi(l-s )$ and $l\leq s+\varphi^{-1}(s)$.

Conversely, assume that $d(T^nw, T^{n+s}w)\leq e^{-(l+1)}$ for some $s\in \N$, $l\in \N$.
Moreover, let $w_s$ be the periodic sequence of period $s$ such that 
\be
\label{periodicSeq}
	w_s(1)\ldots w_s(s) = w(n+1)\ldots w(n+s).
\ee
From $d(T^nw, T^{n+s}w)\leq e^{-(l+1)}$ we obtain 
\bea
\nonumber
	w(n+1 + s)\ldots w(n+1+s + l) &=& w(n+1)\ldots w(n+1+l)
	\\ \nonumber
	&=& w_s(1)\ldots w_s( l)  = w_s(s+1)\ldots w_s(s+1+l),
\eea
where we used \eqref{periodicSeq} if $l\leq s$, and, if $l>s$, say $l= ks + r$, ($k, r \in \N$, $r<s$), that  
\bea
\nonumber
	w(n+1 + is)\ldots w(n+1+(i+1) s ) &=& w(n+1 + (i-1) s)\ldots w(n+1+i s ) = \ldots 
	\\ \nonumber
	&=& w(n+1)\ldots w(n+1+ s ) = w_s(n+1) \ldots w_s(n+s)
	\\ \nonumber
	&=& w_s(n+1 + is)\ldots w_s(n+1+(i+1) s ) 
\eea
for $1 \leq i < k$ and analogously for $i=k$.
This yields
\be
\nonumber
	e^{-(s+l  +1)} \geq d(T^{n}w, w_s)\geq e^{-(s+\varphi^{-1}(s) + 1)},
\ee
using the assumption.
Hence, $s+\varphi^{-1}(s) \geq s+ l$ which shows  $s\geq \varphi(l)$.
\end{proof}

As in \eqref{Bad},
fix a periodic sequence $w_s \in \Sigma$ of  period $s\in \N$ and consider the set 
\bea
	\nonumber
	\textbf{Bounded}_{w_s} &=& 
	\{ w\in \Sigma : \exists \ c=c(w)< \infty \text{ such that } T^n w \not \in B(w_s,  e^{-(s+c +1)}) \text{ for all } n \in \N_0\}
	\\
	\nonumber
	&=& 
	\{ w\in \Sigma : \exists \ l=l(w)< \infty \text{ such that } \frak{p}_{w_s, 1}(T^n w) \leq s+ l  +1 \text{ for all } n \in \N_0 \}.
\eea

From \cite{Weil2}, Theorem 3.8,  we know the following result.
 
\begin{theorem}
The intersection $\bigcap \textbf{Bounded}_{w_s}$ over all periodic $w_s \in \cal{P}_T$ (as well as each particular set $\textbf{Bounded}_{w_s}$)  is a Schmidt-winning set.
\end{theorem}

Due to properties of Schmidt-winning sets, the intersection $\bigcap \textbf{Bounded}_{w_s}$ is nonempty and of Hausdorff-dimension $\log(n)$.
However, this is not sufficient to imply the existence of $\varphi$-aperiodic sequences.
In fact, the above Proposition states that a sequence $w$ is $\varphi$-aperiodic, if and only if 
\be
\nonumber
	w \in \bigcap_{w_s \in \cal{P}_T} \textbf{Bounded}_{w_s}( \varphi),
\ee
where $ \textbf{Bounded}_{w_s}( \varphi)$ is as in \eqref{PenetrationCalF} with $\varphi=G_1$.
Moreover, from Theorem $3.8$ in  \cite{Weil3}, each set $\textbf{Bounded}_{w_s}( \varphi)$ is seen to have Hausdorff-dimension less than $\log(n)$ and such sequences turn out to be extremely rare.

\begin{remark}
Recall that by Proposition \ref{ThmTopEnt} (or by a simple argument), 
every function for which $\varphi$-aperiodic sequences exist
is eventually bounded by $\varphi(l) \leq n^{l+1}=e^{ \log(n)(l+1)}$. 
On the other hand, in view of Theorem \ref{ExistenceWords} and Condition \eqref{BernoulliCriticalNbhd},
consider the function $\varphi_{\d} (l)\equiv e^{\d\log(n)l}$ with $\d<1$ for which $\varphi_{\d}^{-1}(l) = \frac{1}{\d\log(n)}\log(l)$;
hence $s+ \varphi_{\d}^{-1}(s) = s + \frac{1}{\d\log(n)}\log(s)$.
\end{remark}


\subsection{Geodesic flow on hyperbolic manifolds}
\label{SectionGF}
As our central example, we discuss the geodesic flow on hyperbolic manifolds.
Let $M= \H^{n+1} /\Gamma$ be a closed hyperbolic manifold in the following and let $\phi^t$ denote the geodesic flow on the unit tangent bundle $SM$ of $M$;
here, $\H^{n+1}$ denotes the $(n+1)$-dimensional real-hyperbolic space and $\Gamma$ is a torsion-free cocompact lattice in the isometry group of $\H^{n+1}$. 
First we discuss the existence of $G$-aperiodic geodesics and their classification in terms of penetration lengths in neighborhoods of closed geodesics, see Subsection \ref{GeodResults}.
Then we relate the results to a suitable model of Diophantine approximation in negatively curved spaces in Subsection \ref{GeodDA}.
After that we prove the main result of this section in Subsection \ref{GeodProof} and, finally, prove a `metric version' of the closing lemma, see Proposition \ref{ClosingLemma}, in the context of CAT(-1)-spaces in Subsection \ref{GeodClosingLemma}.
\\

\subsubsection{$\varphi$-aperiodic geodesics and main results}
\label{GeodResults}
In the following we identify an orbit $\phi^t(v)$ with the geodesic $\gamma_v \equiv \pi \circ \phi^t(v)$. 
With a slightly different notion using the metrics $d_l$ defined in \eqref{NewBowenMetric}, 
the existence of $\varphi$-aperiodic geodesics follows  from \cite{SchroederWeil}.

\begin{theorem}[\cite{SchroederWeil}, Theorem 4.3]
\label{ExistenceGeodesic}
Assume that $i_M > \log(2) $ and let $\e_0>0$ such that $\log(2) + \e_0 <i_M$. 
Let $\varphi : (1, \infty) \to (\e_0, \infty)$ be a non-decreasing function such that
\be
\nonumber
	\limsup_{l \to \infty} \frac{1}{l} \log(\varphi(l)) < n.
\ee
Then there exists  a length $l_0\geq 0$ and a vector $v \in SM$ 
which satisfies for all times $t_0 \in \R^+$, all lengths $l \geq l_0 $ and shifts $s> \e_0$,  whenever
\be 
\label{MTC}
	d_l( \phi^{t_0}v ,\phi^{t_0 +s}v)  \leq \varepsilon_0 \ \ \ \implies s \geq \varphi(l);
\ee
that is, the ray $\gamma_v$ is a $\varphi$-aperiodic ray with respect to $\e_0$ and $l_0$ (see  \eqref{WeakBarF}).
\end{theorem}

\noindent As remarked in \cite{SchroederWeil} the authors believe that the assumption $i_M> \log(2)$ is not necessary. Moreover, the result holds true in variable negative curvature.

Since a ray is $\varphi$-aperiodic if and only if it is $G$-aperiodic by Lemma \ref{Umrechnung}, we may focus on $\varphi$-aperiodic rays in the following.
Thus, using Theorem \ref{ThmVolEnt}, Theorem \ref{ExistenceGeodesic} and that $h(\phi^t)=\lambda$ by Manning \cite{Manning}, this shows the following.

\begin{corollary}
$\cG = \lambda = h(\phi^t) = n$.
\end{corollary}

\begin{remark} More precisely, in the case of non-positive curvature we have $\cG \leq \lambda = h(\phi^t)$ by  Theorem \ref{ThmVolEnt}, 
with equality in the cases of constant negative curvature by Theorem \ref{ExistenceGeodesic} and zero curvature (since $\lambda=h(\phi^t)=0$ in this case).
This shows Theorem \ref{ThmIntro}.
\end{remark}

Fix $0<\e_0<i_M/2$ as above.
For a closed geodesic $\alpha : \R\to M$, let $v_{\alpha} \equiv \dot \alpha(0)$ be the periodic vector of period $| \alpha |$, the length of $\alpha$.
Given a ray $\gamma=\gamma_v$ in $M$ we adjust the definition of the \emph{penetration length}  $\mathfrak{p}_{\alpha}(v, t_0)$ of $\gamma$ at time $t_0$ in the neighborhood of the closed geodesic $\a$;
that is, set $\mathfrak{p}_{\alpha}(v, t_0) = 0$ if $d(\gamma(t_0), \a) > \e_0/2$, and otherwise 
\be
\nonumber
	\mathfrak{p}_{\alpha} (v, t_0) \equiv \sup\{L\geq 0 : d_L (\phi^{t_0} v, \phi^t v_{\alpha})\leq \e_0/2  \text{ for some time } t \in [0, |\a|] \} \in [0, \infty].
\ee

\begin{remark}
Note that by compactness and local convexity of $\a$ and $d$, we actually have that, 
in other words, if $\mathfrak{p}_{\alpha} (v, t_0) = L \in (0, \infty)$ then there is a time $t \in [0, |\a|]$ such that
\be
\label{DistanceClose}
	d( \gamma(t_0 + s), \alpha( t + s)) \leq \e_0/2 \ \ \text{ for all } s\in [0,L].
\ee
Conversely, if we have \eqref{DistanceClose}, then $\mathfrak{p}_{\alpha} (v, t_0) \geq L$.
Moreover, note that while the $\e_0$-neighborhood of $\a$ might cover $M$, we have $\mathfrak{p}_{\alpha} (v, t_0)=\infty$ if and only if $\gamma_v$ is positively asymptotic to $\a$.
\end{remark}

We will next classify $\varphi$-aperiodic geodesics in terms of their penetration lengths in the neighborhoods of closed geodesics.
More precisely, fix a point $o\in M$ and for a closed geodesic $\a$ in $M$ define the set
\bea
\nonumber
	\textbf{Bounded}_{\a} &\equiv& \{v \in SM_o: \exists L< \infty \text{ such that } \mathfrak{p}_{\alpha} (v, t_0) \leq L \text{ for all } t_0 \geq 0 \},
\eea
as well as, for $L<\infty$, the subset
\bea
\nonumber
	\textbf{Bounded}_{\a}(L) &\equiv& \{v \in SM_o : \mathfrak{p}_{\alpha} (v, t_0) \leq L \text{ for all } t_0 \geq 0 \}.
\eea
The next result follows from \cite{Weil2}, Section $3.6$, where the penetration length is defined slightly differently.
\begin{theorem}
The intersection $\bigcap_\a \textbf{Bounded}_\a$ over all closed geodesics $\a$ in $M$ (as well as each particular set $\textbf{Bounded}_\a$)  is a Schmidt-winning set.
\end{theorem}

In particular, the set  $\bigcap_\a \textbf{Bounded}_\a$ is of Hausdorff-dimension $n=$ dim$(SM_o)$ by properties of Schmidt winning sets.
On the other hand, from Theorem $1.5$ in  \cite{Weil3}, each set $\textbf{Bounded}_{\a}(L)$ follows to have Hausdorff-dimension less than $n$ and such vectors turn out to be extremely rare.

Letting $\d_0 = \log(1+ \sqrt{2})$ we may classify $\varphi$-aperiodic rays by the following Theorem, the main result of this section.

\begin{theorem}
\label{ThmClassificationGeodesics}
Let $v \in SM_o$, $\varphi : (1, \infty) \to (\e_0, \infty)$ be an increasing invertible function and fix $\e_0>0$.
Then, if $\gamma_v$ is $\varphi$-aperiodic (say for the parameter $l_0=0$), then
\be
\label{ConditionGeod}
	v \in \bigcap_{\a \text{ is a closed geodesic in }M}\textbf{Bounded}_{\a}( |\a| + \varphi^{-1}(|\a| )).
\ee
Conversely, if $v$ satisfies condition \eqref{ConditionGeod} and if  $i_M \geq 2\d_0 $, 
then there exists a constant $c_0=c_0(M, \e_0)>0$ as well as  a minimal length $l_0> 2c_0 + 2\e_0$ and a shift $s_0>\e_0$ 
such that $\gamma_v$ satisfies \eqref{MTC} for all shifts $s \geq s_0$ and lengths $l\geq l_0 $ 
for the function
\be
\nonumber
   \tilde \varphi(l) = \varphi( l - 2c_0 - 2\e_0) - \e_0.
\ee
\end{theorem}

\noindent The second assertion states that $\gamma_v$ is $\tilde \varphi$-aperiodic up to the restriction that $s\geq s_0$.
We remark that, due to a technical argument given in \cite{SchroederWeil}, this implies $\gamma_v$ to be $\bar \varphi$-aperiodic for a suitable function $\bar \varphi$ and a length $\tilde l_0 \geq l_0$.

\begin{remark}
If  the parameter $l_0>0$ in the first assertion, then simply replace $\varphi^{-1}(|\a| )$ by $\max\{ l_0, \varphi^{-1}(|\a| )\}$ in \eqref{ConditionGeod}.
We also believe that, again, the requirement $i_M \geq 2\d_0 $ is only a technical requirement and  can be removed.

Recall that by \eqref{IndependentEps}, any function $\varphi$ for which $\varphi$-aperiodic geodesics exist is bounded by $\varphi(l)\leq c\cdot e^{nl}$, where $c=c(M, \e_0)$.
On the other hand, in view of Theorem \ref{ExistenceGeodesic} and Condition \eqref{ConditionGeod},
consider the function $\varphi_{\d} (l)\equiv e^{\d n l}$ with $\d<1$ for which $\varphi_{\d}^{-1}(l) = \frac{1}{\d n}\log(l)$.
\end{remark}

Before proving the Theorem, we connect the above to the theory of Diophantine approximation in negatively curved spaces.
\\

\subsubsection{Diophantine approximation in $\H^{n+1}$}
\label{GeodDA}
For a more general setting, Hersonsky, Parkkonen and Paulin, see \cite{HersonskyPaulin,HersonskyPaulin2,Parkkonen2},
developed a model of Diophantine approximation in the context of negatively curved spaces.
In order to relate the concept of $\varphi$-aperiodic geodesics to this model, and to keep things simple and short, 
let us refer to  \cite{HersonskyPaulin,HersonskyPaulin2,Parkkonen2} for details and only remark the following:
in fact, for our setting let us consider a ray $\gamma$ in $\H^{n+1}$, a lift of a ray $\gamma_v$ in $M$,  starting in a base point $o \in \H^{n+1}$, and the collection $\cC \equiv \{\a_k : k \in \Z\}$ of lifts of a fixed closed geodesic $\a$ in $M$.
Each lift $\a_k$ determines a distance $h_k \equiv d(o, \a_k)$, called a \emph{height}, and two points $\partial_\infty \a_k =\{\a_k(-\infty), \a_k(\infty)\}$, called \emph{resonant points}, in the visual boundary $\partial_\infty \H^{n+1}=S^n$ of $ \H^{n+1}$, where the collection $\partial_\infty \cC \equiv \{\partial_\infty \a_k  \}$ gives a dense set in $S^n$.
The collection $(\partial_\infty \cC, \{h_k\})$ of resonant points and heights gives rise to a model of Diophantine approximation in $S^n$.
We will make use only of a dynamical correspondence in their model, that is,
properties of the point $\xi \equiv \gamma(\infty) \in S^n$ in terms of approximation by resonant points in $\partial_\infty \cC$ can be expressed in terms of the penetration lengths of $\gamma$ in the $\e$-neighborhoods $\cN_\e(\a_k)$ (a convex connected set) of the lifts $\a_k \in \cC$.
In particular, if each penetration length $p_k \equiv | \gamma(\R^+) \cap\cN_\e(\a_k) |$ of $\gamma$ in $\cN_\e(\a_k)$ is bounded by a constant $L<\infty$, then $\xi$ is called \emph{badly approximable}.

Hence, in our setting above, define for $v\in SM_o$,
\be
\nonumber
	\frak{p}_\a(v) \equiv \sup_{t_0 \geq 0} \mathfrak{p}_{\alpha} (v, t_0) \in [0, \infty],
\ee
which determines  the \emph{approximation constant} $c_{\a}(v) \equiv e^{- \frak{p}(v)}$  of $v$ with respect to the closed geodesic $\a$. 
It follows from \cite{HersonskyPaulin2} that for almost all $v \in SM_o$ (spherical measure) we have $c_\a(v)= 0$ (and $\frak{p}_\a(v) = \infty$).
However, by the above, the set \textbf{Bad}$_\a$ of vectors $v$ with $c_{\a}(v) >0$, which are called \emph{badly approximable}, is of Hausdorff-dimension $n$.
Moreover, if \textbf{Bad}$_\a(c)$ is the set of vectors with $c_\a(v)\geq c>0$, then Condition \eqref{ConditionGeod} reads
\be
\nonumber
	v \in \bigcap_{\a \text{ is a closed geodesic in }M}\textbf{Bad}_{\a}( e^{- (|\a| + \varphi^{-1}(|\a| )} )
\ee
and Theorem \ref{ThmClassificationGeodesics} may be stated in terms of this condition.
\\


\subsubsection{Proof of Theorem \ref{ThmClassificationGeodesics}}
\label{GeodProof}

The first assertion of the theorem is straightforward, uses the idea of the proof of Theorem \ref{BoundedLength}, and we only need to be careful with the adjusted definition of the penetration times.
We therefore skip the details.

For the second part, 
let us recall that $M=\H^{n+1}/\Gamma$, where $\Gamma$ is a cocompact torsion-free discrete subgroup of the isometry group of the hyperbolic space $\H^{n+1}$.
Note that every isometry $\psi \in \Gamma$  is of hyperbolic type and can be written as $\psi = \psi_0^k$, $k \in \Z$, with $\psi_0$ primitive.
Every $\psi $ determines an axis $A_\psi=A_{\psi_0}$ (the unique geodesic line which is invariant under $\psi$) and hence a closed geodesic $\alpha= A_{\psi_0}/\Gamma_0 \subset M$,
where $\Gamma_0 \equiv  \langle \psi_0 \rangle \subset \Gamma$.
Conversely, every lift of a closed geodesic in $M$ is an axis and determines a $\psi_0 \in \Gamma$ (or $\psi_0^{-1}$) as above.
For a hyperbolic isometry $\psi \in \Gamma$, denote by $\lvert \psi \rvert \equiv d(x, \psi(x))$, for any $x \in A_\psi$, the translation length of $\psi$ along its axis.
Moreover, for $\e>0$ anda geodesic line $\sigma : \R\to \H^{n+1}$,
note that the $\e$-neighborhood $\cal{N}_{\e}(\sigma)$ of $\sigma$ is a connected convex set.

We  first need to establish a `metric version' of the closing lemma, which implies the strong $\d$-closing property in our context.
Up to the authors' knowledge, this version does neither exist explicitly nor follows easily from a result in the literature so far.
Note that $\H^{n+1}$ is a proper geodesic CAT(-1) space and that the lemma holds even in this setting.

\begin{proposition}[Metric Closing Lemma]
\label{ClosingLemma}
Given $\e_0>0$, there exist a constant $c_0=c_0(\e_0) \leq 2\d_0 + \e_0 - \log(\e_0/8)$,
 a minimal shift $s_0 = 4\e_0 + 6\d_0 $ and a minimal length $l_0=l_0(\e_0) \geq 4\d_0+ \e_0$ such that $s_0 + l_0 \geq 2c_0$ with the following property:

For $l\geq l_0$ and $s>s_0 $, let $\gamma : [0,s+l] \to \H^{n+1}$ be a geodesic segment such that,
\be
\nonumber 
	d(\gamma(s+t),\psi (\gamma(t)) \leq \e_0 \ \ \ \ \text{ for all } t\in [0,l],
\ee
where $\psi \in \Gamma$ is of hyperbolic type with $\lvert \psi \rvert \geq 4 \d_0$.
Then, $s-2\e_0 \leq \lvert \psi \rvert \leq s+\e_0$ and 
\be
\label{ContainedInNbh}
	\gamma([c_0, s+l-c_0]) \subset \cal{N}_{\e_0/8}(A_{\psi}).
\ee
\end{proposition}

Assuming the Proposition for the moment, we are able prove the Theorem.
We want to verify \eqref{MTC} in our context.
Therefore, let $\gamma$ be a lift of $\gamma_v$ and  assume that 
\be
\nonumber
	d(\gamma(t_0+s+t), \psi( \gamma(t_0 + t ))\leq \e_0 \ \ \ \ \forall t\in [0,l]
\ee
for a shift $s>s_0$ and length $l>l_0 + 2 c_0 + 2\e_0$ ($s_0$ and $l_0$ as in the Proposition above) and some $\psi\in \Gamma$.
Note that $\psi \in \Gamma$ is of hyperbolic type
with  $\lvert \psi \rvert  \geq 2i_M \geq 4 \d_0$.
Applying Proposition \ref{ClosingLemma}  we get that $s-2\e_0 \leq \lvert \psi \rvert \leq s+\e_0$ and 
\be
	\nonumber
	\gamma([t_0 + c_0, t_0 + s+l-c_0]) \subset \cal{N}_{\e_0/8}(A_{\psi}).
\ee
It is readily shown that this implies 
\be
\nonumber 
	d( \gamma(t_0 + c_0 + t) , A_{\psi}(\bar t_0 + t)) \leq \e_0/2
\ee 
for at least all $t \in [0, L]$ with $L \geq s+l-2c_0 - \e_0$ (for a suitable parametrization of $A_\psi$ and a time  $\bar t_0$).
Using Conditon \eqref{ConditionGeod} we get
\be
\nonumber
	s+ l - 2c_0 - \e_0 \leq L \leq \lvert \psi \rvert +  \varphi^{-1}(\lvert \psi \rvert ),
\ee
and since $\lvert \psi \rvert \leq s+\e_0$, we obtain that
\bea
\nonumber
	s \geq \lvert \psi \rvert - \e_0 \geq  \varphi( l - 2c_0 - 2\e_0) - \e_0 =  \tilde \varphi(l).
\eea
This finishes the proof of the Theorem, up to replacing $l_0$ by $l_0 + 2c_0 + 2\e_0$.
\\


\subsubsection{Proof of the Metric Closing Lemma}
\label{GeodClosingLemma}
In order to prove the proposition, we need the following Lemmata.
In the following,  $Z$ denotes a proper geodesic CAT(-1) space and we let $\delta_0$ be the constant such that $Z$ is a $\d_0$-hyperbolic space;
recall that $\H^{n+1}$ is  $\log(1 + \sqrt{2})$-hyperbolic.
For two points $x$ and $y \in Z$, let $[x, y]$ (identified with its image) denote the unique geodesic segment from $x$ to $y$.

\begin{lemma}[\cite{Weil2}, Lemma 3.19]
\label{Neighbor}
Let $D \geq \e>0$.
Let $\gamma$ and $\alpha$ be two geodesics in $Z$ such that $d(\gamma(-L) , \alpha) \leq D$ and $d(\gamma(L), \alpha) \leq D$, 
where $L\geq 2(D - \log(\e))$.
Then there exists a constant $c=c(D, \e) \leq D - \log(\e)$ such that  $\gamma([-L+c, L-c]) \subset \cal{N}_{\e}(\alpha)$.
\end{lemma}

Moreover, we need estimates for the displacement function $d_{\psi}(x) \equiv  d(x, \psi(x))$.

\begin{lemma}\label{Translation}
For $\psi \in \Gamma$ hyperbolic with $\lvert \psi \rvert \geq 4 \delta_0$  
and $x\in Z$, we have
\be
\nonumber
	\max\{ 2d(x, A_{\psi}), \lvert \psi \rvert \} - 4\delta_0 \leq d_{\psi}(x) \leq \lvert \psi \rvert + 2d(x, A_{\psi}).
\ee
\end{lemma}

\begin{proof}
Note that if $pr : X \to A_{\psi}$ denotes the closest point projection on the convex closed set $A_{\psi}$, then $pr(\psi (x)) = \psi (pr(x))$. 
Hence, $d(x, A_{\psi}) = d(\psi(x), A_{\psi})$ and $d(pr(x), pr(\psi(x))= \lvert \psi \rvert$.
Therefore, the upper bound follows easily.

Let $m \in [pr(x),pr(\psi(x))]$ such that $d(m, pr(x))= \lvert \psi \rvert/2$.
Note that if $m$ is $\delta_0$-close to $[x, pr(x)]$, then $\lvert \psi \rvert/2 = d(pr(x), m) < \delta_0$.
Hence, assume there is a point  $\bar m \in [x, pr(\psi(x))]$  which is $\delta_0$-close to $m$. 
If $\bar m$ is in turn $\delta_0$-close to $[x, \psi x]$, say to the point $\bar x$, then let $d_1=d(x, \bar x)$ and $d_2=d(\bar x, \psi(x))$.
Considering the triangle $(x, pr(x), m)$ with $\angle_{pr(x)}(x, m)\geq \pi/2$, we have
\bea 
\nonumber
	d_1 		&\geq& 	d(x, m) - 2 \delta_0	 
	\\ \nonumber
			&\geq& \max\{ d(x, pr(x)), d(pr(x), m) \} - 2\delta_0 = \max\{ d(x, A_{\psi}), \lvert \psi \rvert/2 \} - 2\delta_0.
\eea 
The same lower bound holds for $d_2$ which shows the claim in this case.

If there exists no such point $\bar x$, then $\bar m$ is $\delta_0$-close to a point $\bar y$ in $[pr(\psi(x)), \psi(x)]$
and, since $\angle_{pr(\psi(x))}(\bar y, m) \geq \pi/4$, we have $\lvert \psi \rvert/2= d(m, pr(\psi(x)) \leq d(m, \bar y) < 2\delta_0$.
\end{proof}

\begin{proof}[Proof of Proposition \ref{ClosingLemma}]
Set $x=\gamma(0)$, $y=\gamma(l)$ and $z=\gamma(s+l)$.
Let $n\in \N$ be the minimal integer such that $\lvert \psi^n \rvert = n \lvert \psi \rvert \geq 4\delta_0$.
Note that since $\lvert \psi \rvert \geq2 i_M>0$ we have $n \leq \lceil 2\d_0 /i_M \rceil$.
By Lemma \ref{Translation}, 
\bea 
\nonumber
	n(s + \e) &\geq&  n(d(\gamma(0), \gamma(s)) + d(\gamma(s), \psi(\gamma(0)) ) 
		\\ \nonumber
		&\geq& n d_{\psi}(x) \geq d_{\psi^n}(x)
		\\ \nonumber
		&\geq& \max\{ 2d(x, A_{\psi}), n\lvert \psi \rvert \} -4 \d_0 \geq 2d(x, A_{\psi}) -4 \d_0 .
\eea
Hence, $d(x, A_{\psi}) \leq D_s = D(s, \e, i_M)$ and analogously, $d(y, A_{\psi}) \leq D_s$.
Moreover, 
\be
\nonumber
	d(\gamma(s+l), A_{\psi}) \leq d(\gamma(s+l), \psi(\gamma(l)) + d(\psi(\gamma(l)), A_{\psi}) \leq D_s + \e.
\ee
Thus, using that $\lvert \psi \rvert \geq 4\d_0$, hence $n=1$, we already have for $d_1\equiv d(\gamma(0), A_{\psi})$ and $d_2 \equiv d(\gamma(s+l), A_{\psi})$ the bounds
\be
\label{FirstBound}
	d_i \leq D_s + \e \leq \frac{1}{2} (s+\e) + 2\d_0 + \e = s/2 + 2\d_0 + \e/2 \equiv s/2 + c_1,
\ee
and we claim, when $s >  2c_1 + 2\d_0 = s_0$,  that $d_i \leq 2\d_0 + \e$.

In fact, note first that, up to reversing the orientation of $\gamma$ and using $\psi^{-1}$ instead of $\psi$ (and consider the reversed situation), we may assume that $d_1 \geq d_2$.
Let $x_s \equiv \gamma(s)$ and denote by $pr : Z \to A_{\psi}$ again the closest point projection onto $A_{\psi}$.
Considering the geodesic triangle $(x, pr(x), z)$, 
for every point  $y \in[x, pr(x)]$ we have using \eqref{FirstBound} that 
\be
\nonumber
	d(x_s,y) \geq d(x_s,x ) - d(x, y) \geq s - d_1 \geq s - s/2 - c_1 > \d_0,
\ee
and there must exists a point $\bar x_s \in [pr(x), z]$ with $d(x_s, \bar x_s) \leq \d_0$. 
Now consider the geodesic triangle $(pr(x), pr(z), z)$.
If there exists a point on $[pr(x), pr(z)]$ which is $\d_0$-close to $\bar x_s$, hence $d(x_s, A_{\psi}) \leq 2\d_0$, we have 
\be
\nonumber
	d_2 \leq d_1 = d(x, A_{\psi})= d(\psi(x), A_{\psi}) \leq d(\psi(x), x_s) + d(x_s, A_{\psi}) \leq \e + 2\d_0,
\ee
hence the claim.
Thus, let $m \in [z, pr(z)]$ be a point with $d(\bar x_s, m) \leq \d_0$.
Clearly, this requires
\be
\nonumber
	d(z, m) \geq d(z, x_s) - 2\d_0 = l-2\d_0.
\ee
However, as above, we have
\bea
	\nonumber
	d_1 &\leq& d(\psi(x), x_s) + d(x_s, A_{\psi}) 
	\\ \nonumber	
	&\leq& d(x_s, m) +   d(m, pr(z)) +  \e  
	\\ \nonumber
	&\leq&  2\d_0 + \e + (d_2 - d(m, z))  \leq 2\d_0 + \e + d_1 - (l - 2\d_0),
\eea
which is a contradiction whenever $l\geq 4\d_0 + \e$. This finishes the claim.

Hence, by Lemma \ref{Neighbor}, there exists a constant $c_0 \leq 2\d_0 + \e  - \log(\e/8)$ 
such that \eqref{ContainedInNbh} holds.
Finally, since 
\be	
\nonumber
	s-\e \leq d_{\psi}(\gamma( c_0)) \leq \lvert \psi \rvert + 2d(\gamma(c_0),A_{\psi}) \leq \lvert \psi \rvert + \e,
\ee
and  $\lvert \psi\rvert \leq d_{\psi}(x) \leq s+\e$,
we have $s- 2\e \leq \lvert \psi \rvert \leq  s+\e$.
This finishes the proof.
\end{proof}


\bibliographystyle{amsplain}

\bibliography{cup_ref.bib}

\end{document}